\documentclass[12pt]{amsart}
\usepackage[english]{babel}
\usepackage{paralist,url,verbatim, anysize}
\usepackage{amscd}
\usepackage[all,cmtip]{xy} 
\usepackage{amsmath,amsthm,amssymb,amsfonts}
\usepackage[bookmarksopen=true]{hyperref}
\usepackage[usenames, dvipsnames]{color}
\usepackage{graphicx}
\usepackage{hyperref}
\usepackage{euler, amsfonts, amssymb, latexsym, epsfig,epic}

\makeatletter
\let\@fnsymbol\@arabic
\makeatother

\theoremstyle{plain}
\newtheorem{theorem}{\bf Theorem}[section]
\newtheorem{example1}[theorem]{\bf Example}
\newtheorem{conjecture}[theorem]{Conjecture}

\newtheorem{corollary}[theorem]{Corollary}

\newtheorem{proposition}[theorem]{Proposition}

\newtheorem{question}[theorem]{Question}
\newtheorem{remark}[theorem]{Remark}

\theoremstyle{definition}

\newtheorem{definition}[theorem]{Definition}

\newtheorem*{theorem*}{\bf Theorem}

\newcommand{\sd}{\operatorname{sd} }
\newcommand{\cd}{\operatorname{cd} }
\newcommand{\init}{\operatorname{in} }
\newcommand{\gin}{\operatorname{gin} }

\newcommand{\N}{\mathbb{N}}
\newcommand{\Z}{\mathbb{Z}}

\newcommand{\C}{\mathcal{C}}

\newcommand{\CC}{\mathbb{C}}

\newcommand{\Min}{\operatorname{Min} }

\newcommand{\height}{\operatorname{height} }
\newcommand{\Ker}{\operatorname{Ker} }
\newcommand{\Coker}{\operatorname{Coker} }

\newcommand{\Proj}{\operatorname{Proj} }
\newcommand{\Spec}{\operatorname{Spec} }
\newcommand{\Ext}{\mathrm{Ext}}
\newcommand{\depth}{\operatorname{depth} }

\newcommand{\pp}{\mathfrak{p}}

\newcommand{\ttt}{\mathfrak{t}}
\newcommand{\mm}{\mathfrak{m}}
\newcommand{\nn}{\mathfrak{n}}

\newcommand{\Hom}{\operatorname{Hom} }
\newcommand{\homo}{\operatorname{hom} }
\newcommand{\reg}{\operatorname{reg} }

\definecolor{mypink}{RGB}{215, 5, 234}

 \newcommand{\GCD}{\mathrm{GCD}}
 \newcommand{\LCM}{\mathrm{LCM}}
  \newcommand{\GL}{\mathrm{GL}}
\newcommand{\chara}{\mathrm{char}}

\begin{document}
\title{Square-free Gr\"obner degenerations}
\author{Aldo Conca}
\email{conca@dima.unige.it}
\address{Dipartimento di Matematica, Universit\'a di Genova, Italy}
\author{Matteo Varbaro} 
\email{varbaro@dima.unige.it}
\address{Dipartimento di Matematica, Universit\'a di Genova, Italy} 
 \thanks{Both authors are supported by PRIN 2010S47ARA 003, `Geometria delle Variet\`a Algebriche.'} 
  \date{}
\maketitle

\begin{abstract}
Let $I$ be a homogeneous ideal of  $S=K[x_1,\dots, x_n]$ and let  $<$ be  a term order. We prove that if the initial ideal  $J=\init_<(I)$ is radical then the extremal Betti numbers of $S/I$ and of $S/J$ coincide. In particular,   $\depth(S/I)=\depth(S/J)$ and $\reg(S/I)=\reg(S/J)$.
\end{abstract}

\section{Introduction}
Let $S$ be the  polynomial ring $K[x_1,\dots,x_n]$ over a field $K$ equipped with its standard graded structure. Let  $M$ be a graded $S$-module. We denote by $\beta_{ij}(M)$ the $(i,j)$-th Betti number of  $M$, and by $h^{ij}(M)$ the dimension of the degree $j$ component of  its $i$-th local cohomology module $H_\mm^i(M)$ supported on the maximal ideal $\mm=(x_1,\dots,x_n)$. Furthermore  $\reg(M)$ denotes the  Castelnuovo-Mumford regularity of $M$ and $\depth(M)$ its depth. Finally recall that a non-zero Betti number $\beta_{i,i+j}(M)$ is called extremal  if $\beta_{h,h+k}(M)=0$ for every $h\geq i$, $k\geq j$ with  $(h,k)\neq (i,j)$.

Let $I$ be a homogeneous ideal of  $S$. For every term order $<$ on $S$ we may associate to $I$ a monomial ideal $\init_<(I)$ via the computation of a Gr\"obner basis.  For  notational simplicity and when there is no danger of confusion we   suppress the dependence of the monomial ideal on $<$ and denote it by  $\init(I)$. The ideal $\init(I)$ is called the initial ideal of $I$ with respect to $<$ and  $S/\init(I)$ can be realized as the special fiber of a flat family whose generic fiber is $S/I$. This process is called a Gr\"obner degeneration. 
If the ideal $I$ is in generic coordinates (see \cite[Chapter 15.9]{Ei1} for the precise definition)  the outcome of the  Gr\"obner degeneration is called generic initial ideal of $I$ with respect to $<$ and it is denoted by $\gin(I)$. 

 It is well known that the homological and cohomological invariants behave well under  Gr\"obner degenerations.  
 Indeed one has: 

$$\beta_{ij}(S/I)\leq \beta_{ij}(S/\init(I)) \quad \mbox{ and}  \quad  h^{ij}(S/I) \leq  h^{ij}(S/\init(I)) \mbox{ for all } i,j  $$
and, in particular, 
$$\reg(S/I)\leq \reg(S/\init(I)) \quad \mbox{ and }  \quad  \depth(S/I)\geq \depth(S/\init(I)).$$ 

Simple examples show that the inequalities are in general strict. 
On the other hand, Bayer and Stillman \cite{BS} proved that equality  holds  is a special and important case: 

\begin{theorem}[Bayer-Stillman]
\label{BaSt}
 Let  $<$ be the degree reverse lexicographic order. Then for every homogeneous ideal $I$ of $S$ one has: 
 $$\reg(S/I)=\reg(S/\gin(I))  \mbox{ \ and \ } \depth(S/I) = \depth(S/\gin(I)).$$
\end{theorem}

Furthermore Bayer, Charalambus and Popescu \cite{BCP} generalized  Theorem \ref{BaSt}  proving that $S/I$ and $S/\gin(I)$ have the same extremal Betti numbers (positions and values). 
\medskip  
 
Algebras with straightening laws  (ASL for short)  were introduced by  De Concini,
Eisenbud and Procesi  in \cite{DEP, Ei} and, more or less simultaneously and in a slightly different way,  by Baclawski \cite{Ba}. Actually they appear under the name of ``ordinal
Hodge algebras'' in \cite{DEP}  while the terminology ``ASL''  is used  by Eisenbud  in his survey  \cite{Ei} and later on by Bruns and Vetter in \cite{BV}. 
This notion arose  as an axiomatization of  the underlying combinatorial structure observed by 
many authors  in  classical algebras appearing in  invariant theory, commutative algebra  and algebraic geometry. For example,   coordinate rings of flag varieties,  their  Schubert subvarieties  and various kinds of rings defined by determinantal equations   are ASL.   Roughly speaking an ASL  is an algebra $A$ whose generators and relations are governed by  a finite poset $H$ in a  special way. Any ASL $A$   has a discrete counterpart   $A_D$ defined  by square-free monomials of degree $2$.  Indeed it turns out that  $A_D$ can be realized as Gr\"obner degeneration of $A$ and ASL's can be characterized via Gr\"obner degenerations  \cite[Lemma 5.5]{Co}.

In \cite{DEP}  it is proved, in the  general setting of Hodge algebras,  that  $A$ is Gorenstein or Cohen-Macaulay if $A_D$ is  Gorenstein or Cohen-Macaulay and the authors mention that: ``The converse is false: It is easy for $A$ to be Gorenstein without $A_D$ being so, and presumably the same could happen for Cohen-Macaulayness''.  Within the general  framework of Hodge algebras  discussed in  \cite{DEP},  a Cohen-Macaulay algebra  with a non-Cohen-Macaulay associated discrete  algebra   has been provided by Hibi in \cite[Example pg. 285]{Hi}. In contrast, a Cohen-Macaulay ASL with non-Cohen-Macaulay discrete counterpart could not be found. In this paper we show that such an example does not exist because $\depth A=\depth A_D$ for any ASL $A$ (Corollary \ref{ASLCM}).   

The question whether some homological invariants of an ASL depend only on its discrete counterpart leads  quickly to consider possible generalizations.  Herzog was guided  by these considerations to conjecture the following:   

\begin{conjecture}[Herzog]
\label{Herzogcon}
Let $I$ be a homogeneous ideal of a standard graded polynomial ring $S$, and $<$ a term order on $S$. If $\init(I)$ is square-free, then the extremal Betti numbers of $S/I$ and those of $S/\init(I)$ coincide (positions and values).
\end{conjecture}

In other words, Herzog's intuition was  that a square-free initial ideal behaves, with respect to the homological invariants,  as the reverse lexicographic  generic initial ideal. However the stronger statement asserting that $\gin(I)=\gin(\init(I))$ if $\init(I)$ is square-free turned out to be false, see Example \ref{ingin}. 
The above conjecture, in various forms, has been discussed in several occasions by Herzog and his collaborators.  It  appeared in print only recently in \cite[Conjecture 1.7]{CDG4} and in the introduction of \cite{HR18}.  In this paper we solve positively Herzog's conjecture. Indeed we establish a stronger result: 

\begin{theorem}
\label{the1}
Let $I$ be a homogeneous ideal of a standard graded polynomial ring $S$ such that $\init(I)$ is a square-free monomial ideal for some term order. Then
\[h^{ij}(S/I)=h^{ij}(S/\init(I)) \ \ \ \forall \ i, j. \]
\end{theorem}

Since the extremal Betti numbers of $S/I$ can be characterized in terms of  the values of  $h^{ij}(S/I)$, Herzog's conjecture  follows from Theorem \ref{the1}.  We note that Theorem \ref{the1} holds  for more general gradings, see  Remark \ref{r:grad}.  

It turns out that, in many respects,  the relationship between $I$ and a square-free initial ideal $\init(I)$ (when it exists) is tighter than the relation between $I$ and its degree reverse lexicographic generic initial ideal $\gin(I)$. 
For example Herzog and Sbarra proved in \cite{HS} that  if $K$ has characteristic $0$ then  the assertion of Theorem \ref{the1} with $\init(I)$ replaced by  $\gin(I)$ holds if and only if $S/I$ is sequentially Cohen-Macaulay. Furthermore, a consequence of Theorem \ref{the1} is that, if  $\init(I)$ is a square-free monomial ideal, then $S/I$ satisfies Serre's condition $(S_r)$ if and only if $S/\init(I)$ does (see Corollary \ref{c:2}); the latter statement is false for generic initial ideals. As a further remark note that  Conjectures 1.13 and 1.14 in \cite{CDG4} are indeed special cases of  Theorem \ref{the1}. 

\bigskip

The paper is structured as follows. Section \ref{s:main} is devoted to prove Theorem \ref{the1}, and to draw some immediate consequences. In Section \ref{s:sqg} we discuss properties of  ideals admitting a monomial square-free initial ideal and some consequences of Theorem \ref{the1} on them.  We discuss as well three families of ideals  with square-free initial ideals that appeared in the literature:  ideals defining ASL's (\ref{s:ASL}), Cartwright-Sturmfels ideals (\ref{s:cs}) and Knutson ideals (\ref{s:knu}).  Finally, in Section \ref{s:q}, we discuss some open questions.

\bigskip

{\it Acknowledgements}. The authors wish to express their gratitude to Linquan Ma for enlightening discussions concerning the results in \cite{KK}, as well as for sharing the preprint \cite{DDM}. Several examples appearing in the paper have been discovered by means of computations performed with the computer algebra systems \cite{cocoa,m2}.

\section{The main result}\label{s:main}

The goal of the section is to prove Theorem \ref{the1}. We use the notation of the introduction. In the examples appearing in this paper, the term orders will always refine the order of the variables $x_1>\ldots >x_n$. Furthermore, we will write ``lex'' for the lexicographic term order, and ``degrevlex'' for the degree reverse lexicographic term order.

The main ingredient of the proof is Proposition \ref{p:crucial} that  
can be regarded as a characteristic free version of   \cite[Theorem 1.1]{KK}. 
One key ingredient in the proof of Proposition \ref{p:crucial} is the notion of cohomologically full singularities that  was  introduced and studied in the recent preprint  \cite{DDM}. Let us recall the definition: 

\begin{definition}
\label{cofull}
A Noetherian local ring $(A,\nn)$ is cohomologically full if it satisfies the following condition. 
For every local ring $(B,\mm)$  such that $\chara(B)=\chara(A)$ and $\chara(B/\mm)=\chara(A/\nn)$ 
 and every  surjection of local rings $\phi:(B,\mm)\rightarrow (A,\nn)$ such that the induced map $\overline{\phi}:B/\sqrt{(0)}\rightarrow A/\sqrt{(0)}$  is an isomorphism, one has that the induced map on local cohomology $H^i_{\mm}(B)\rightarrow H^i_{\mm}(A)=H^i_{\nn}(A)$ is surjective for all $i\in\N$.
\end{definition}

We point out that the  closely related notion of  liftable local cohomology  is introduced and  discussed  by Koll\'ar and Kov\'acs in the paper \cite{KKb} which is a  generalized and expanded version of  \cite{KK}.   The next proposition is inspired by  \cite[Proposition 5.1]{KK}.

\begin{proposition}\label{p:art}
Let $(R,\mathfrak{t})$ be an Artinian local ring, and $(A,\nn)$ be a Noetherian local flat $R$-algebra such that the special fiber $A/\mathfrak{t}A$ is cohomologically full. Let $N$ be a finitely generated $R$-module, and $N=N_0\supseteq N_1\supseteq\ldots \supseteq N_q\supseteq N_{q+1}=0$ a filtration of submodules such that $N_j/N_{j+1}\cong R/\ttt$ for all $j=0\ldots ,q$. Then, for all $i\in\N$ and $j=0\ldots ,q$, the following complex of $A$-modules is exact:
\[0\rightarrow H_{\nn}^i(N_{j+1}\otimes_{R}A)\rightarrow H_{\nn}^i(N_{j}\otimes_{R}A)\rightarrow H_{\nn}^i((N_j/N_{j+1})\otimes_{R}A)\rightarrow 0\]
\end{proposition}
\begin{proof}
Notice that the surjection $A\xrightarrow{\phi} A/\ttt A$ yields an isomorphism between $A/\sqrt{(0)}$ and $ (A/\ttt A)/\sqrt{(0)}$. Since the tensor product is right-exact, we have a surjection of $A$-modules 
\[N_{j}\otimes_{R}A\xrightarrow{\beta} (N_j/N_{j+1})\otimes_{R}A\cong A/\ttt A.\]
Denoting by $\beta'$ the composition of $\beta$ with the isomorphism $(N_j/N_{j+1})\otimes_{R}A\cong A/\ttt A$, choose $x\in N_{j}\otimes_{R}A$ such that $\beta'(x)=1$, and set $\alpha:A\rightarrow N_{j}\otimes_{R}A$ the multiplication by $x$. Then $\beta'\circ \alpha:A\rightarrow A/\ttt A$ equals $\phi$. Therefore, being $A/\ttt A$ cohomologically full, the induced map of $A$-modules
\[H^k_{\nn}(\beta'\circ \alpha)=H^k_{\nn}(\beta')\circ H^k_{\nn}(\alpha):H^k_{\mathfrak{n}}(A)\rightarrow H^k_{\mathfrak{n}}(A/\ttt A)\]
is surjective for all $k\in\N$, so that $H^k_{\nn}(\beta):H^k_{\nn}(N_{j}\otimes_{R}A)\rightarrow H^k_{\nn}((N_j/N_{j+1})\otimes_{R}A)$ is surjective as well. Since $A$ is a flat $R$-algebra, for each $j=0,\ldots ,q$ we have a short exact sequence of $A$-modules
\[0\rightarrow N_{j+1}\otimes_{R}A\rightarrow N_{j}\otimes_{R}A\xrightarrow{\beta} (N_j/N_{j+1})\otimes_{R}A\rightarrow 0.\]
Passing to the long exact sequence on local cohomology
\begin{eqnarray*}
\ldots \to H^{i-1}_{\nn}(N_{j}\otimes_{R}A)\xrightarrow{H^{i-1}_{\nn}(\beta)} H^{i-1}_{\nn}((N_j/N_{j+1})\otimes_{R}A)\rightarrow \\
H^i_{\nn}(N_{j+1}\otimes_{R}A)\rightarrow H^i_{\nn}(N_{j}\otimes_{R}A)\xrightarrow{H^i_{\nn}(\beta)} H^i_{\nn}((N_j/N_{j+1})\otimes_{R}A)\to \ldots ,
\end{eqnarray*}
being each $H^k_{\nn}(\beta)$ surjective, we get the thesis.
\end{proof}

The following is, essentially, already contained in \cite{Lyu}.

\begin{proposition}\label{p:lyu}
Let  $J\subseteq S=K[x_1,\dots,x_n]$ be a square-free monomial ideal. Then $(S/J)_{\mm}$ is cohomologically full.
\end{proposition}
\begin{proof}
If $K$ has positive characteristic, then $(S/J)_{\mm}$ is $F$-pure by \cite[Proposition 5.38]{HR}. Hence, in characteristic zero $(S/J)_{\mm}$ is of $F$-pure type (and therefore of $F$-injective type). 
Then $(S/J)_{\mm}$ is Du Bois by \cite[Theorem~6.1]{schwede}.
So in each case, we conclude by \cite[Lemma 3.3, Remark 3.4]{MSS}.
\end{proof}

In the proposition below, $(R,\ttt)$ is a homomorphic image of a Gorenstein local ring, $P=R[x_1,\ldots ,x_n]$ is a standard graded polynomial ring over $R$ and $A$ is a graded quotient of $P$. Denote by $\nn$ the unique homogeneous maximal ideal $\ttt P+(x_1,\ldots ,x_n)$ of $P$. 

\begin{proposition}\label{p:crucial}
With the notation above, assume furthermore that $A$ is a flat $R$-algebra. If $(A/\ttt A)_{\nn}$ is cohomologically full, then $\Ext^i_{P}(A,P)$ is a free $R$-module for all $i\in \Z$. 
\end{proposition}
\begin{proof}
Let  $X=\Spec(A_{\nn})$ and $Y=\Spec(R)$. Then $f:X\to Y$ is a flat, essentially of finite type morphism of local schemes which is embeddable in a Gorenstein morphism. Furthermore, if $y\in Y$ is the closed point, the fiber $X_y$ is the affine scheme $\Spec(A_{\nn}/\ttt A_{\nn})$. In \cite[Corollary 6.9]{KK} it is proved that, if in addition the schemes are essentially of finite type over $\CC$ and $X_y$ has Du Bois singularities, then $h^{-i}(\omega^{\bullet}_{X/Y})$ is flat over $Y$ for any $i\in\Z$, where $\omega^{\bullet}_{X/Y}$ denotes the relative dualizing complex of $f$ (see \cite[Section 45.24]{SP} for basic properties of relative dualizing complexes). 
We note  that the proof of  \cite[Corollary 6.9]{KK} holds as well by replacing the assumption that  $X_y$ has  Du Bois singularities with the (weaker) assumption that $X_y$ has cohomologically full singularities. In fact, the Du Bois assumption is used only in the proof of \cite[Proposition 5.1]{KK}, that as we noticed in Proposition \ref{p:art} holds true even under the assumption that the special fiber is cohomologically full.

So, also under our assumptions, we have that $h^{-i}(\omega^{\bullet}_{X/Y})$ is flat over $Y$ for any $i\in\Z$. But $\omega^{\bullet}_{X/Y}$ is the sheafication of $R\Hom(A_{\nn},P_{\nn})[n]$, so
$\Ext^{n-i}_{P_{\nn}}(A_{\nn},P_{\nn})\cong \Ext^{n-i}_{P}(A,P)_{\nn}$ is a flat $R$-module. So $\Ext^{n-i}_{P}(A,P)$ is a flat $R$-module, and therefore $\Ext_P^{n-i}(A,P)_j$ is a flat $R$-module for all $j\in\Z$. Being $\Ext_P^{n-i}(A,P)$ finitely generated as $P$-module, $\Ext_P^{n-i}(A,P)_j$ is actually a finitely generated flat, and so free, $R$-module for any $j\in\Z$. In conclusion,   $\Ext_P^{n-i}(A,P)$ is a direct sum of free $R$-modules and hence it is a free $R$-module itself.
\end{proof}

We are ready to prove Theorem \ref{the1}:
\begin{proof}[Proof of Theorem \ref{the1}]
Set  $J=\init(I)$ and let   $w=(w_1,\ldots ,w_n)\in \N^n$ be a weight such that $J=\init_w(I)$ (see, for example, \cite[Proposition 1.11]{sturmfels}). Let $t$ be a new indeterminate over $K$, $R=K[t]_{(t)}$ and $P=R[x_1,\ldots ,x_n]$. Provide $P$ with the grading given by
$\deg(x_i)=1$ and $\deg(t)=0$. By considering the $w$-homogenization $\homo_w(I)\subseteq P$, set $A=P/\homo_w(I)$. It is well known that the inclusion $R\to A$ is flat, $A/(t)\cong S/J$ and $A\otimes_RK(t)\cong (S/I)\otimes_KK(t)$. In particular, by Proposition \ref{p:lyu}, the special fiber $(A/(t))_{\nn}$, where $\nn$ is the unique homogeneous maximal ideal of $P$, is cohomologically full. By Proposition \ref{p:crucial}, hence, $\mathrm{Ext}_P^i(A,P)$ is a free $R$-module for any $i\in\Z$. So $\Ext_P^i(A,P)_j$ is a finitely generated free $R$-module for any $j\in\Z$. Say $\Ext_P^i(A,P)_j\cong R^{r_{i,j}}$. Since $t\in P$ is a nonzero divisor on $A$ we have the short exact sequence 
\[0\to A\xrightarrow{\cdot t}A\rightarrow A/(t)\rightarrow 0.\]
Applying $\Hom_P(-,P)$ to it, for all $i\in \Z$ we get:
\[0\to \Coker(\alpha_{i,t})\rightarrow\Ext_P^{i+1}(A/(t),P)\rightarrow \Ker(\alpha_{i+1,t})\rightarrow 0,\]
where $\alpha_{k,t}$ is the multiplication by $t$ on $\Ext_S^k(A,P)$. Notice that there are natural isomorphisms $\Ext_P^{i+1}(A/(t),P)\cong \Ext_S^i(A/(t),P)$ (cf. \cite[Lemma 3.1.16]{BH}).
We thus have, for all $i\in\Z$, the short exact sequence:
\[0\to \Coker(\alpha_{i,t})\rightarrow\Ext_S^i(A/(t),S)\rightarrow \Ker(\alpha_{i+1,t})\rightarrow 0.\]
For any $j\in\Z$, the above short exact sequence induces a short exact sequence of $K$-vector spaces
\[0\to \Coker(\alpha_{i,t})_j\rightarrow\Ext_S^i(A/(t),S)_j\rightarrow \Ker(\alpha_{i+1,t})_j\rightarrow 0.\]
Since $\Ext_P^k(A,P)_j\cong R^{r_{k,j}}$, we get
\[\Ker(\alpha_{i+1,t})_j=0 \ \ \ \mbox{and} \ \ \ \Coker(\alpha_{i,t})_j\cong K^{r_{i,j}}. \]
Therefore 
$\Ext_S^i(A/(t),S)_j\cong K^{r_{i,j}}$. On the other hand 
\begin{eqnarray*}
(\Ext_{S}^i(S/I,S)_j)\otimes_KK(t)\cong \Ext_{S\otimes_KK(t)}^i((S/I)\otimes_KK(t),S\otimes_KK(t))_j\cong \\
\Ext_{P\otimes_RK(t)}^i(A\otimes_RK(t),P\otimes_RK(t))_j\cong (\Ext_P^i(A,P)_j)\otimes_RK(t)\cong K(t)^{r_{i,j}}.
\end{eqnarray*}
Therefore $\Ext_S^i(S/J,S)_j\cong K^{r_{i,j}}\cong \Ext_{S}^i(S/I,S)_{j}$ for all $i,j\in\Z$. By Grothendieck local duality we have:
\[h^{ij}(S/I)=\dim_K\Ext_S^{n-i}(S/I,S)_{j-n} \mbox{ \ \ \ and \ \ \ } h^{ij}(S/J)=\dim_K\Ext_S^{n-i}(S/J,S)_{j-n}\]
for all $i,j\in\Z$, so we conclude.
\end{proof}

\begin{remark}\label{r:grad}
The proof of Theorem  \ref{the1}   works also for more general gradings.  Assume $S=K[x_1,\dots, x_n]$ is equipped with a $\Z^m$-graded structure such that $\deg(x_i)\in \N^m\setminus \{0\}$.  Let $I\subseteq S$ be  a $\Z^m$-graded  ideal such that $\init(I)$ is square-free, then
\[\dim_K H^i_{\mm}(S/I)_v=\dim_K H^i_{\mm}(S/\init(I))_v \ \ \ \forall \ i\in\N, \ v\in \Z^m.\]
\end{remark}

\begin{remark}
 Theorem \ref{the1}  holds as well  under the assumption that  $S/\init(I)$ is  cohomologically full. 
 Examples of cohomologically full rings arise form known ones via flat extensions. For example, for a sequence $a=a_1, \ldots ,a_n$ of positive integers  one can consider the $K$-algebra map $\phi_a:S\to S$  defined by $\phi_a(x_i)=x_i^{a_i}$, which is indeed a flat extension.  Since cohomologically fullness is preserved under flat extensions (see \cite[Lemma 3.4]{DDM}) one has that $S/\phi_a(J)$ is cohomologically full if $S/J$ is cohomologically full. In particular, by Proposition \ref{p:lyu} we have that $S/\phi_a(J)$ is cohomologically full if $J$ is a square-free monomial ideal. 
Hence the conclusion of Theorem \ref{the1} holds also if $\init(I)=\phi_a(J)$ where $J$ is square-free and $a$ is any sequence of positive integers. 
\end{remark}

Since  the extremal Betti numbers of $S/I$ can be described in terms of $h^{ij}(S/I)$ (cf. \cite{Ch}), we have:

\begin{corollary}
\label{c:1}
Let $I\subseteq S$ be a homogeneous ideal such that $\init(I)$ is square-free. Then the extremal Betti numbers of $S/I$ and those of $S/\init(I)$ coincide  (positions and values). In particular, $\depth S/I=\depth S/\init(I)$ and  $\reg S/I=\reg S/\init(I)$.
\end{corollary}

\begin{remark}\label{r:cmrad}
One could wonder if $S/\sqrt{\init(I)}$ is Cohen-Macaulay whenever $S/I$ is Cohen-Macaulay (independently from the fact that $\init(I)$ is square-free).  In \cite{Va1} it is proved that, if $S/I$ is Cohen-Macaulay,  $\Proj S/\init(I)$ cannot be disconnected by removing a closed subset of codimension larger than $1$ (a necessary condition for the Cohen-Macaulayness of $S/\sqrt{\init(I)}$). However, we present two determinantal examples,  one  for lex and the other for revlex, such that $S/I$ is Cohen-Macaulay but $S/\sqrt{\init(I)}$ not.  
 \begin{itemize} 
 \item[(1)] 
 Let $S=K[x_1,\ldots ,x_7]$ and $I$ be the ideal of $2$-minors of :
\[\begin{pmatrix}
x_1+x_2 & x_5 & x_4 \\
-x_5+x_6 & x_3+x_7 & x_5 \\
x_4+x_7 & x_1-x_3 & x_5+x_7
\end{pmatrix}
\]
It turns out that $S/I$ is a $3$-dimensional Cohen-Macaulay domain.   With respect to  lex, one has: 
\[\init(I)=(x_1x_5, x_4x_5, x_1x_6, x_1x_4, x_1x_7, x_1^2, x_3x_4, x_2x_5, x_1x_3, x_2x_6x_7, x_2x_4x_7, x_3x_5^2, x_2x_3x_6).\]
One can check  that $\depth S/\sqrt{\init(I)}=2$.

\item[(2)]  Let $S=K[x_1,\ldots ,x_9]$ and $I$ the ideal of 2-minors of :

\[\begin{pmatrix}
x_3+x_7 & x_6 & x_1 & x_5 \\
x_9 & x_4+x_5 & x_7 & x_1+x_2 \\
x_3 & x_3 & x_7 & x_7-x_8
\end{pmatrix}
\]
It turns out  that $S/I$ is a $3$-dimensional Cohen-Macaulay ring.  With respect to  revlex one has: 
\begin{eqnarray*}
\init(I)= (x_1x_7, x_1x_3, x_3x_7, x_4x_7, x_4x_8, x_3x_4, x_2x_7, x_3x_5, x_2x_3, x_6x_7, x_3^2, x_1^2, x_4x_5,x_1x_4,x_7x_8, \\
x_7^2, x_3x_9,x_5x_6x_9, x_4x_6x_9, x_2x_6x_9, x_1x_6x_9, x_1x_8^2, x_1x_5x_9,x_1x_6x_8, x_5x_6x_8, x_2x_6x_8, x_2x_5x_8^2x_9).
\end{eqnarray*}
One can check  that $\depth S/\sqrt{\init(I)}=2$.
\end{itemize} 
\end{remark}

Corollary \ref{c:1} and  Theorem \ref{BaSt}  show that  a square-free initial ideal and the revlex generic initial ideal have important features in common. 
Next we show that from other perspectives a square-free initial ideal (when it exists!) behaves better than the revlex generic initial ideal. 
We recall first some definitions.  Let $A=S/J$ where  $J\subseteq S$  is  a  homogeneous ideal. 
\begin{compactenum}
\item $A$ is Buchsbaum if for any homogeneous system of parameters $f_1,\ldots ,f_d$ of $A$,
\[(f_1,\ldots ,f_{i-1}):f_i=(f_1,\ldots ,f_{i-1}):\mm \ \ \ \forall \ i=1,\ldots ,d.\]
\item $A$ is generalized Cohen-Macaulay  if $H_{\mm}^i(A)$ has finite length for all $i<\dim A$.
\item For $c\in \N$, $A$ is Cohen-Macaulay in codimension $c$ if $A_{\pp}$ is Cohen-Macaulay for any prime ideal $\pp$ of $A$ such that $\height \pp\leq \dim A-c$.
\item For $r\in\N$, $A$ satisfies Serre's $(S_r)$ condition if $\depth A_{\pp}\geq \min\{r,\height \pp\}$ for any prime ideal $\pp$ of $A$.
\end{compactenum}

\begin{remark}
If $J\subseteq S$ is an ideal and $\pp\subseteq S$ is a prime ideal of height $h$ containing $J$, then for all $k\in\N$:
\begin{align*}
\depth(S_{\pp}/JS_{\pp})\geq k & \iff H_{\pp S_{\pp}}^i(S_{\pp}/JS_{\pp})=0 \ \ \forall \ i<k \\
& \iff \mathrm{Ext}^{h-i}_{S_{\pp}}(S_{\pp}/JS_{\pp},S_{\pp})=0 \ \ \forall \ i<k \\
& \iff (\mathrm{Ext}^{h-i}_{S}(S/J,S))_{\pp}=0 \ \ \forall \ i<k.
\end{align*}
Recall that $S/J$ is pure if $\dim S/\pp=\dim S/J$ for all associated prime ideal $\pp$ of $J$. Since
\[\dim \mathrm{Ext}^k_{S}(S/J,S)=\sup\{n-\height \pp: \ \pp\in\Spec S, \ (\mathrm{Ext}^k_{S}(S/J,S))_{\pp}\neq 0\},\]from the equivalences above can deduce that:
\begin{compactenum}
\item $S/J$ is pure $\iff$ $\dim \mathrm{Ext}^{n-i}_{S}(S/J,S)<i \ \ \forall \ i<\dim S/J$.
\item $S/J$ is generalized Cohen-Macaulay $\iff$ $\dim \mathrm{Ext}^{n-i}_{S}(S/J,S)\leq 0 \ \ \forall \ i<\dim S/J$.
\item $S/J$ is Cohen-Macaulay  in codimension $c$ $\iff$ $\dim \mathrm{Ext}^{n-i}_{S}(S/J,S)<c \ \ \forall \ i<\dim S/J$.
\item $S/J$ satisfies $(S_r)$ for $r\geq 2$ $\iff$ $\dim \mathrm{Ext}^{n-i}_{S}(S/J,S)\leq i-r \ \ \forall \ i<\dim S/J$.
\end{compactenum}
For the latter equivalence, the assumption $r\geq 2$ is needed to guarantee the purity of $S/J$,  see \cite[Lemma (2.1)]{Sc79} for details. 
\end{remark}

\begin{remark}
Given a homogeneous ideal $J\subseteq S$, if $A=S/J$ is Buchsbaum then it is generalized Cohen-Macaulay, but the converse does not hold true in general. If $H_{\mm}^i(A)$ is concentrated in only one degree for all $i<\dim A$, however, $A$ must be Buchsbaum by \cite[Theorem 3.1]{Sc}. If $J$ is a square-free monomial ideal and $A$ is generalized Cohen-Macaulay, then it turns out that $H_{\mm}^i(A)=(H_{\mm}^i(A))_0$ for all $i<\dim A$. In particular, for a square-free monomial ideal $J\subseteq S$ we have that $S/J$ is Buchsbaum if and only if $S/J$ is generalized Cohen-Macaulay. 
\end{remark}

Putting together the remarks above and Theorem \ref{the1} we get:

\begin{corollary}\label{c:2}
Let $I\subseteq S$ be a homogeneous ideal such that $\init(I)$ is square-free. Then
\begin{compactitem}
\item[{\rm (i)}] $S/I$ is generalized Cohen-Macaulay if and only if $S/I$ is Buchsbaum if and only if $S/\init(I)$ is Buchsbaum;
\item[{\rm (ii)}] For any $r\in\N$, $S/I$ satisfies the $(S_r)$ condition if and only if $S/\init(I)$ satisfies the $(S_r)$ condition;
\item[{\rm (iii)}] For any $c\in\N$, $S/I$ is Cohen-Macaulay in codimension $c$ if and only if $S/\init(I)$ is Cohen-Macaulay in codimension $c$.
\end{compactitem}
\end{corollary}

\section{Ideals with square-free initial ideals}\label{s:sqg}

In this section, we will outline some properties and non-properties of  ideals with square-free initial ideals. After that we will show three special classes of these ideals appeared in the literature, pointing at some consequences of Theorem \ref{the1} for each of them.  

We first observe that by \cite{BCP} two ideals with the same degrevlex generic initial ideal have the same extremal Betti numbers. Hence a stronger version of  Herzog's conjecture \ref{Herzogcon} would be  the assertion that if $I$ is an ideal with a square-free initial ideal $J$  then the corresponding degrevlex generic initial ideals coincide. 
The following example  shows that the stronger statement is actually false.

\begin{example1}
\label{ingin}
Let $S=K[x_{ij}: 1\leq i,j,\leq 4]$  and $<$ be the degrevlex order associated to the total order $x_{11}>x_{12}>x_{13}>x_{14}>x_{21}>\dots >x_{44}$. Let  $I$ be the ideal of $2$-minors of  $(x_{ij})$.   Then $J=\init(I)=(x_{ij}x_{hk} : i<h \mbox{ and } j>k)$  is square-free and quadratic  and  $\gin(I)$ and $\gin(J)$ differ already in degree $2$. Indeed,  degree $2$ part   of $\gin(I)$ is  $( x_{ij} : i=1,2 \mbox{ and } j=1,2,3,4)^2$ and the degree  $2$ part of $\gin(J)$ is obtained from that of $\gin(I)$ by replacing $x_{23}x_{24}, x_{24}^2$ 
 with $x_{11}x_{31}, x_{12}x_{31}$. 
\end{example1}

When $K$ has positive characteristic then  $S/J$ is $F$-pure for any square-free monomial ideal $J\subseteq S$ (cf. \cite[Proposition 5.38]{HR}). However, $S/I$  may fail $F$-purity when $I$ has a  square-free initial ideal:

\begin{example1}\label{ex:f-pure}
Let $S=K[x_1,\ldots, x_5]$ where $K$ has characteristic $p>0$, and $I$ the ideal generated by the $2$-minors of the matrix:
\[\begin{pmatrix}
x_4^2+x_5^a & x_3 & x_2 \\
x_1 & x_4^2 & x_3^b-x_2
\end{pmatrix}.\]
Note that, if $\deg(x_4)=a$, $\deg(x_1)=\deg(x_3)=1$, $\deg(x_2)=b$ and $\deg(x_5)=2$, the ideal $I$ is homogeneous. Singh proved in \cite[Theorem 1.1]{Si} that, if $a-a/b>2$ and $\mathrm{GCD}(p,a)=1$, then $S/I$ is not $F$-pure. However, considering lex as term order, one has
\[\init(I)=(x_1x_3, \ x_1x_2, \ x_2x_3).\]
\end{example1}

On the other hand we have: 

\begin{proposition}\label{p:full}
Let $I\subseteq S$ be a homogeneous ideal with a with square-free initial ideal. Then $S/I$ is cohomologically full. 
\end{proposition}
\begin{proof}
As in the proof of Theorem \ref{the1}, setting $J=\init(I)$, take a weight $w=(w_1,\ldots ,w_n)\in \N^n$ such that $J=\init_w(I)$. Let $t$ be a new indeterminate over $K$, $R=K[t]_{(t)}$ and $P=R[x_1,\ldots ,x_n]$. By considering the $w$-homogenization $\homo_w(I)\subseteq P$, set $A=P/\homo_w(I)$. Since $A/(t)\cong S/J$ is cohomologically full and $t$ is a non-zero divisor on $A$, then $A$ is cohomologically full by \cite[Theorem 3.1]{DDM}. So, $A_t\cong (S/I)\otimes_KK(t)$ is cohomologically full by \cite[Lemma 3.4]{DDM}, and therefore (again by \cite[Lemma 3.4]{DDM}), $S/I$ is cohomologically full.
\end{proof}

\begin{corollary}
Let $I\subseteq S$ be a   homogeneous ideal with a  square-free initial ideal. If $J\subseteq S$ is a homogeneous ideal such that $\sqrt{J}=I$, then
\[\depth S/J\leq \depth S/I \ \ \ \mbox{ and } \ \ \ \reg(S/J)\geq \reg(S/I).\]
\end{corollary}
\begin{proof}
It follows by Proposition \ref{p:full}.
\end{proof}

Given an ideal $I\subseteq S$, the cohomological dimension of $I$ is defined as:
\[\cd(S,I)=\max\{i\in\N:H^i_I(S)\neq 0\}.\]
One may ask what is the relationship  between the cohomological dimension of $I$ and that of $\init(I)$. In general, without  further assumptions on $\init(I)$,   they are unrelated. 
\begin{example1}
Recall that for any ideal $J\subseteq S$ of height $h$ and generated by $r$ polynomials $h\leq \cd(S,J)\leq r$ and furthermore, $\cd(S,J)=\cd(S,\sqrt{J})$.
\begin{enumerate}
\item For any ideal $I\subseteq S$ of height $h$, the generic initial ideal $\gin(I)$ w.r.t. degrevlex has cohomological dimension $h$: in fact $\sqrt{\gin(I)}=(x_1,\ldots, x_h)$. However,  there are many ideals $I$ of height $h$  for which $\cd(S,I)> h$.  Hence there are ideals for which 
\[\cd(S,I)>\cd(S,\init(I)).\]
holds. 
\item In \cite[Example 2.14]{Va1} has been considered the ideal \[I=(x_1x_5+x_2x_6+x_4^2, \ x_1x_4+x_3^2-x_4x_5, \ x_1^2+x_1x_2+x_2x_5)\subseteq S=K[x_1,\ldots ,x_6].\]
It turns out that $I$ is a height $3$ complete intersection and $S/I$ is a normal domain (notice that in \cite[Example 2.14]{Va1} there is a typo in the last equation). By considering lex, we have:
\[\sqrt{\init(I)}=(x_1,x_2,x_3)\cap (x_1,x_3,x_6)\cap (x_1,x_2,x_5)\cap (x_1,x_4,x_5).\]
However $S/\sqrt{\init(I)}$ has depth $2$, so $\cd(S,\sqrt{\init(I)})=6-2=4$ by \cite{Lyu}. Therefore $\cd(S,\init(I))=4$, but $\cd(S,I)=3$. So,  there are ideals for which  
\[\cd(S,I)<\cd(S,\init(I)).\]
holds. 
\end{enumerate}
\end{example1}

\begin{proposition}
\label{cdim} 
If $I\subseteq S$ is a homogeneous ideal and $\init(I)$ is square-free, then
\[\cd(S,I)\geq \cd(S,\init(I)).\]
Furthermore, if $K$ has positive characteristic, then $\cd(S,I)= \cd(S,\init(I))$.
\end{proposition}
\begin{proof}
Under this assumption $S/I$ is cohomologically full by Proposition \ref{p:full}, so $\cd(S,I)\geq n-\depth(S/I)$ by \cite[Proposition 2.5]{DDM}.
However, $\depth(S/I)=\depth(S/\init(I))$ by Corollary \ref{c:1}, and $\cd(S,\init(I))=n-\depth(S/\init(I))$ by \cite{Lyu}.

If $K$ has positive characteristic, $\cd(S,I)\leq n-\depth(S/I)$ by \cite[Proposition 4.1 and following remark]{PS}.
\end{proof}

\begin{example1}
If $K$ has characteristic 0, the inequality in Proposition  \ref{cdim}   may be strict.  For example if $S=K[X]$ where $X=(x_{ij})$ denotes an $r\times s$ generic matrix, the ideal $I\subseteq S$ generated by the size $t$ minors of $X$ has cohomological dimension $rs-t^2+1$ by a result of Bruns and Schw\"anzl in \cite{BrSc}. However, Sturmfels proved in \cite{St} that, if $<$ is lex refining $x_{11}>x_{12}>\ldots >x_{1s}>x_{21}>\ldots >x_{r1}>\ldots >x_{rs}$, then $\init(I)$ is square-free and $S/\init(I)$ is Cohen-Macaulay of dimension $rs-(r-t+1)(s-t+1)$. In particular, once again by \cite{Lyu}, $\cd(S,\init(I))=(r-t+1)(s-t+1)<\cd(S,I)$.
\end{example1}

\subsection{ASL: Algebras with straightening laws}\label{s:ASL} 

Let $A=\oplus_{i\in \N} A_i$ be a graded algebra and let $(H,\prec)$ be a finite poset set. 
 Let $H\to \cup_{i>0}A_i$ be an  injective function. The elements of $H$ will be identified with their images. Given a chain $h_1\preceq h_2 \preceq \dots  \preceq h_s$ of elements of $H$ the corresponding product   $h_1\cdots h_s\in A$ is called  standard monomial. 
 One says that $A$ is an algebra with straightening laws on $H$ (with respect to the given embedding $H$ into $\cup_{i>0}A_i$) if three conditions are satisfied: 
 \begin{compactenum}
\item The elements of $H$ generate  $A$ as a $A_0$-algebra. 
 \item The standard monomials are $A_0$-linearly independent. 
 \item For every pair $h_1, h_2$ of incomparable elements of $H$ there is a relation (called the straightening law) 
 $$h_1h_2=\sum_{j=1}^u  \lambda_j h_{j1}\cdots h_{jv_j}$$ 
 where $\lambda_j\in A_0\setminus \{0\}$,  the  $h_{j1}\cdots h_{jv_j}$  are distinct standard monomials and, assuming that $h_{j1}\preceq \dots \preceq  h_{jv_j}$,   one has $h_{j1}\prec h_1$ and $h_{j1}\prec h_2$ for all $j$. 
 \end{compactenum}
 It then follows  from the three axioms that the standard monomials form a basis of $A$ over $A_0$ and that the 
 straightening laws are indeed the defining equations of $A$ as a quotient of the polynomial ring $A_0[H]=A_0[h : h\in H]$. That is, the kernel $I$ of the canonical surjective map $A_0[H]\to A$ of $A_0$-algebras induced by the function $H\to \cup_{i>0}A_i$ is generated by the straightening laws regarded as elements of $A_0[H]$, i.e., 
\[A=A_0[H]/I \quad \mbox{ with } \quad  I=( h_1h_2-\sum_{j=1}^u  \lambda_j h_{j1}\cdots h_{jv_j} : h_1\not\prec h_2 \not\prec h_1).\]
\begin{remark}
Equipping the polynomial ring  $A_0[H]$ with the $\N$-graded structure induced by assigning to $h$ the degree of its image in $\cup_{i>0}A_i$, then the ideal $I$ is homogeneous.
\end{remark}

The ideal $J=( h_1h_2 : h_1\not\prec h_2 \not\prec h_1)$ of $A_0[H]$ defines a quotient $A_D=A_0[H]/J$ which is an ASL as well, called the discrete ASL associated to $H$.  In \cite{DEP}  it is proved that  $A_D$ is the special fiber of a flat family with general fiber $A$. Indeed, at least when $A_0$ is a field,  one can obtain the same result  by 
 observing that with respect to (weighted) degrevlex associated to a total order on $H$ that refines the given partial order $\prec$ one has $J=\init(I)$. More precisely it  has been observed in \cite[Lemma 5.5]{Co} that   ASL's  can also be defined via  Gr\"obner degenerations.

As we have already said in the introduction,  various kinds of generic determinantal  rings are indeed ASL's. Here it is important to note that the ASL presentation is not, in general, the minimal presentation of the algebra. For example, for generic determinantal rings  the poset consists  of all the subdeterminants of the matrix. Nevertheless the non-minimal presentation is good enough to prove, via deformation to the discrete counterpart, that these classical algebras are Cohen-Macaulay.  As a consequence of Theorem \ref{the1} and Remark \ref{r:grad} we have: 

\begin{corollary} 
\label{ASLCM}
Let $A$ be an ASL over a field $K$ and let $A_D$ be the corresponding discrete counterpart. 
Then $\depth A=\depth A_D$. In particular,  $A$ is Cohen-Macaulay if and only if  $A_D$ is Cohen-Macaulay. 
\end{corollary} 

\begin{remark}
In \cite{Mi},   Miyazaki proved that if $A$ is a Cohen-Macualay ASL and  $A_D$ is Buchsbaum then $A_D$ is Cohen-Macualay. 
\end{remark}

\begin{remark} Let  $S=K[x_1,\dots,x_n]$ and let $H$ be the set of square-free monomials different from $1$  ordered by division: for $m_1,m_2\in H$ one sets  $m_1\preceq m_2$ if and only if $m_2|m_1$. Then $S$ can be regarded as an ASL over $H$ with straightening law: 
$$m_1m_2=\GCD(m_1,m_2)\LCM(m_1,m_2).$$
This induces an ASL structure on  every Stanley-Reisner ring $K[\Delta]$ associated with a simplicial complex  $\Delta$ on $n$ vertices. Here the underlying poset is given by the non-empty faces of $\Delta$ ordered by reverse inclusion.  Hence the  discrete ASL associated to $K[\Delta]$ is $K[\sd(\Delta)]$ where $\sd(\Delta)$ is the barycentric subdivision of $\Delta$. So, denoting by $N$ the number of square-free monomials different from $1$ of $S$ and by $S'=K[x_m:m\neq 1 \mbox{ is a square-free monomial of }S]$, by Theorem \ref{the1} and by Grothendieck duality we have
\[\dim \mathrm{Ext}^{n-i}_{S}(K[\Delta],S)=\dim \mathrm{Ext}^{N-i}_{S'}(K[\sd(\Delta)],S') \ \ \forall \ i\in\Z \]
(Krull dimensions). This was already known: in fact, the Krull dimensions of the deficiency modules of a Stanley-Reisner ring are topological invariants by the work of Yanagawa \cite{Ya}.
\end{remark} 
  
 A further application of our main result is a proof of the Eisenbud-Green-Harris conjecture for Cohen-Macaulay standard graded ASLs. This conjecture has been originally presented  in \cite{EGH} in various ways. Here we refer to  the formulation discussed in the fourth section of  \cite{EGH}. 



\begin{theorem}\label{t:EGH}
Let $A$ be a standard graded Cohen-Macaulay ASL on $H$.  Then any  Artinian reduction of  $A$ verifies Eisenbud-Green-Harris conjecture. In other words, the $h$-vector of $A$ is equal to the $f$-vector of some simplicial complex on $|H|-\dim(A)$ vertices.
\end{theorem}

\begin{proof}
Let $|H|=n$, $S=K[x_1,\ldots ,x_n]$ and write $A=S/I$ where $I$ is the ideal generated by the straightening laws. With respect to a suitable term order the ideal $\init(I)$ is generated by square-free quadratic monomials. Furthermore, $\depth(S/\init(I))=\depth(S/I)$ by Corollary \ref{c:1}. So the conclusion follows using \cite{CCV}, because $S/I$ and $S/\init(I)$ have the same Hilbert function: in fact \cite[Theorem 2.1]{CCV} implies that the $h$-vector of $S/\init(I)$ equals the $h$-vector of $S/J$ where $J$ is a homogeneous ideal containing $(x_1^2,\ldots ,x_{n-d}^2)$ with $S/J$ Cohen-Macaulay, where $d=\dim(A)$. Therefore, considering an Artinian reduction $S'/J'$ of $S/J$, where $S'=K[x_1,\ldots ,x_{n-d}]$, $J'\subset S'$ is still a homogeneous ideal containing $(x_1^2,\ldots ,x_{n-d}^2)$ and the Hilbert function of $S'/J'$ equals the $h$-vector of $A$. If $\init(J')$ is any initial ideal of $J'\subset S'$, then the faces of the desired simplicial complex are those $\sigma \subset \{1,\ldots ,n-d\}$ such that $\prod_{i\in\sigma}x_i$ is not in $\init(J')$. 
\end{proof}

\subsection{Cartwright-Sturmfels ideals}\label{s:cs}
Cartwright-Sturmfels ideals were introduced and studied  by Conca, De Negri, Gorla in a series of recent papers \cite{CDG1,CDG2,CDG3,CDG4} inspired by the work \cite{CS}.  We recall briefly their definition and main properties. Given positive integers  $d_1,\dots, d_m$ one considers the polynomial ring  $S=K[x_{ij} : 1\leq i\leq m \mbox{ and }1\leq j\leq d_i]$ with $\Z^m$-graded structure induced by assignment $\deg(x_{ij})=e_i\in \Z^m$. The group $G=\GL_{d_1}(K)\times \cdots \times \GL_{d_m}(K)$ acts on $S$  as the group of multigraded $K$-algebra automorphisms. The Borel subgroup  $B=B_{d_1}(K)\times \cdots  \times B_{d_m}(K)$ of the upper triangular invertible matrices acts on $S$ by restriction.  An ideal $J$ is  Borel-fixed if $g(J)=J$ for all $g\in B$.  
A multigraded ideal $I\subset S$ is Cartwright-Sturmfels if its multigraded Hilbert function coincides with that of a Borel-fixed radical ideal. The main properties of Cartwright-Sturmfels ideals are: 

\begin{compactenum}
\item[(1)]  If $I$ is Cartwright-Sturmfels then all its initial ideals are square-free. 
\item[(2)]  The set  of Cartwright-Sturmfels ideals is closed under multigraded linear sections and multigraded projections. 
\end{compactenum}

Examples of  Cartwright-Sturmfels ideals are: 

\begin{compactenum} 
\item[(3)] The ideal of $2$-minors and the ideal of maximal minors of matrices of distinct variables with graded structure given by rows or columns, \cite{CDG1, CDG2, CDG3}.
\item[(4)] Binomial edge ideals, \cite{CDG4}.
\item[(5)] Ideals of multigraded closure of linear spaces, \cite{CDG4}. 
\end{compactenum} 

So, multigraded linear sections and multigraded projections of the above ideals Cartwright-Sturmfels as well. 
Note that Conjecture \cite[1.14]{CDG4} turns out to be a special case of the multigraded version of  Theorem \ref{the1}, see Remark \ref{r:grad}. 

\subsection{Knutson ideals}\label{s:knu}
For this subsection, either $K=\mathbb{Q}$ or $K=\Z/p\Z$. Fix also $f\in S=K[x_1,\dots,x_n]$ such that $\init(f)$ is a square-free monomial for some term order $<$. 

Let $\C_f$ be the smallest  set of ideals of $S$ satisfying the following conditions:
\begin{compactenum}
\item $(f)\in\C_f$;
\item If $I\in\C_f$, then $I:J\in\C_f$ for any ideal $J\subseteq S$; in particular the associated primes of $I$ are in $\C_f$. 
\item If $I,J\in\C_f$, then $I+J\in\C_f$ and $I\cap J\in\C_f$.
\end{compactenum} 

\begin{example1}
If $f=x_1\cdots x_n$,  then $\C_f$ is the set of all the  square-free monomial ideals of $S$. 
\end{example1}

The class of ideals $\C_f$ is introduced and  studied by Knutson in \cite{Kn}. For this reason, we will call the elements of $\C_f$  Knutson ideals associated with $f$. In positive characteristic, if $I\in \C_f$ then $\phi_f(I)\subseteq I$ where $\phi_f:S\to S$ is a special  splitting (associated to the polynomial $f$) of the Frobenius morphism $F:S\to S$. In particular, if $I\in \C_f$, then $S/I$ is $F$-pure (in positive characteristic).
Knutson proved in \cite{Kn} that for every $I\in \C_f$ the initial ideal  $\init(I)$ is  square-free and hence every ideal in $\C_f$ is radical.  Moreover, one can infer that  $\init(I)\neq \init(J)$ whenever $I,J\in\C_f$ are different. 
So  $\C_f$ is always a finite set.

In \cite[Section 7]{Kn}, Knutson has shown that many interesting ideals belong to  $\C_f$ for a suitable choice of $f$, for example
ideals defining matrix Schubert varieties and ideals defining Kazhdan-Lusztig varieties. Below we present  an example of a Knutson prime ideal that is not Cohen-Macaulay. The interest lais in the fact that we expect that a prime ideal with a square-free initial ideal, under some additional assumption, should be Cohen-Macaulay (see Section \ref{s:q}). This example originated from a discussion we had with Jenna Rajchgot at MSRI in 2012; we thank her for pointing us that the example we discussed there is actually a Knutson ideal:

\begin{example1}\label{ex:notCM}
Let $S=K[x_1,\ldots ,x_5]$, and $f=gh$ where $g=x_1x_4x_5-x_2x_4^2-x_3x_5^2$ and $h=x_2x_3-x_4x_5$. By considering lex, $\init(f)=x_1x_2x_3x_4x_5$. We have that $(g, h)$ is a height 2 complete intersection in $\C_f$. One can check that 
\[\pp=(g, h, \ \ x_1x_3x_4-x_3^2x_5-x_4^3, \ \ x_1x_2x_5-x_2^2x_4-x_5^3)\]
is a height 2 prime ideal. Since it contains $(g,h)$, must be associated to it, so $\pp\in\C_f$ and one can check  that $S/\pp$ is not Cohen-Macaulay.
\end{example1}

Next, we show how to derive the solution of a conjecture stated in \cite{BeV} in the case of Knutson ideals. 
Given an ideal $I\subseteq S$, the graph with vertex set $\Min(I)$ and edges $\{\pp,\pp'\}$ if $\height(\pp+\pp')=\height(I)+1$ is called the dual graph of $I$. The ideal $I$ is called Hirsch if the diameter of the dual graph is bounded above from $\height(I)$. In \cite{BeV,DV} has been explained why Hirsch ideals are a natural class to consider, and there have also been provided several examples of such ideals. In particular, in \cite[Conjecture 1.6]{BeV} has been conjectured that, if $I$ is quadratic and $S/I$ is Cohen-Macaulay, then $I$ is Hirsch.
\begin{proposition}
Let $I\subseteq S$ be a homogeneous Knutson ideal such that $S/I$ satisfies $(S_2)$ (e.g. $S/I$ is Cohen-Macaulay). If either $\init(I)$ is quadratic or $\height(I)\leq 3$, then $I$ is Hirsch.
\end{proposition}
\begin{proof}
By the assumption $\init(I)$ is the Stanley-Reisner ideal $I_{\Delta}$ of a simplicial complex $\Delta$ on $n$ vertices. Since $S/I$ satisfies $(S_2)$, $S/I_{\Delta}$ satisfies $(S_2)$ as well by Corollary \ref{c:2}. In other words, $\Delta$ must be a normal simplicial complex. If $I_{\Delta}$ is quadratic (that is $\Delta$ is flag), then $I_{\Delta}$ is Hirsch by \cite{AB}. If $\height(I_{\Delta})\leq 3$, then it is simple to check that $I_{\Delta}$ is Hirsch (see  \cite[Corollary A.4]{Ho} for the less trivial case in which $\height(I_{\Delta})=3$). In each case, we conclude because, under the assumptions of the theorem, by \cite[Theorem 3.3]{DV} the diameter of the dual graph of $I$ is bounded above from that of $\init(I)$.
\end{proof}

\section{Questions and answers }\label{s:q}
The first version of this paper, posted on  arXiv on May 30 2018,  contains in Section 4 five  questions. 
We received several comments concerning these questions and it turns out that  only Question \ref{q:rat}  is still open.   We reproduce below the questions with the original numbering and the relative answers. 

\begin{question}
Let $I\subseteq S$ be a prime ideal with a square-free initial ideal. Does  $S/I$ satisfy  Serre's condition $(S_2)$?
\end{question}
Notice that, if $\pp\subseteq S$ is a prime ideal with a square-free initial ideal, then $\depth S/\pp\geq 2$ (provided $\dim S/\pp\geq 2$) by \cite{KaSt,Va1}. 
However the answer is negative:  the ideal in  \ref{ex:notCM} does not  satisfy  Serre's condition $(S_2)$.

\begin{question}\label{q:rat}
Let $I\subseteq S$ be a homogeneous prime ideal with a square-free initial ideal   such that $\Proj S/I$ is nonsingular. Is  $S/I$ Cohen-Macaulay and with negative $a$-invariant?
\end{question}

A similar question appeared in \cite[Problem 3.6]{Va2}. For ASL's, a question similar to Question \ref{q:rat} already caught some attention in the eighties:  Buchweitz proved a related statement (unpublished, see \cite{DEP}), that using Theorem \ref{the1} can be expressed as follows:

\begin{theorem}[Buchweitz]
Let $A$ be an ASL domain over a field of characteristic $0$ such that $\Proj A$ has rational singularities. If $A$ is Cohen-Macaulay, then $A$ has rational singularities.
\end{theorem}

A result of Brion \cite{Br} on multiplicity free irreducible varieties implies  $S/I$ is Cohen-Macaulay  for every  prime Cartwright-Sturmfels ideal $I$.   Similarly one can ask: 

\begin{question}
\label{KnPrCM}
Let $I\subseteq S$ be a Knutson prime ideal. 
Is  $S/I$  Cohen-Macaulay?   
\end{question}

Jenna Rajchgot  informed us that the ideal Example \ref{ex:notCM}  is indeed a  Knutson prime ideal $\pp\subseteq S$ such  that $S/\pp$ is not Cohen-Macaulay.  This provides a negative answer to Question \ref{KnPrCM}.

\begin{question}
\label{qFpure} 
Let $I\subseteq S$ be an ideal such that $\init(I)$ is a square-free monomial ideal for degrevlex. If $\mathrm{char}(K)>0$, is it true that $S/I$ is $F$-pure?
\end{question}

In Example \ref{ex:f-pure} $S/I$ is not $F$-pure and $I$ is an ideal with a square-free initial ideal. However, in that case $\init(I)$ is a square-free monomial ideal only for lex, so it does not provide a negative answer to the above question. Also, the question whether ASL's over fields of positive characteristic are $F$-pure was already raised in \cite{DEP}, and as explained in \ref{s:ASL}  ASL's have square-free quadratic degenerations  with respect to degrevlex term orders. 

Combining  results on Cartwright-Sturmfels ideals recalled  in \ref{s:cs} with a result of Othani   \cite{Oh13}, we obtain a negative  answer to  Question \ref{qFpure}.  Othani   \cite{Oh13}  proves that, if $I$ is the binomial edge ideal of the $5$-cycle, the ring  $S/I$ is not $F$-pure in characteristic $2$.  As  recalled  in \ref{s:cs} binomial edge ideals, being Cartwright-Sturmfels, have a  square-free monomial ideal for every term order.  
As far as we know it might be true that  for a Cartwright-Sturmfels ideal $I$  the rign  $S/I$ is of $F$-pure type in characteristic $0$.

\begin{question}\label{q:norm}
Let $I\subseteq S$ be a prime ideal such that $\init(I)$ is a square-free monomial ideal for degrevlex. Is it true that $S/I$ is normal?
\end{question}

Without assuming that the term order is of degrevlex type, it is easy to see that the answer to the above question is negative (consider for example  $(x_1x_2x_3+x_2^3+x_3^3)\subseteq K[x_1,x_2,x_3]$). On the other hand,  Eisenbud in \cite{Ei} conjectured a positive answer to Question \ref{q:norm} for ASL's.
Hibi informed  us that a $3$-dimensional standard graded  ASL which is a non-normal (Gorenstein) domain is described in \cite{HW85}.  This provides a negative answer to Question \ref{q:norm}.


\begin{thebibliography}{BMS2}

   \bibitem[CoCoA]{cocoa}    J. Abbott, A.M. Bigatti, L. Robbiano, \emph{CoCoA: a system for doing Computations in Commutative Algebra}, available at \url{http://cocoa.dima.unige.it}.

\bibitem[AB14]{AB}
K.~Adiprasito, B.~Benedetti, {\it The Hirsch conjecture holds for normal flag complexes}, Math. Oper. Res. 39, 1340--1348, 2014.

\bibitem[Ba81]{Ba} K.~Baclawski {\it Rings with lexicographic straightening law}, Adv. Math.  39, 185--213 1981.
 

\bibitem[BCP99]{BCP} D.~Bayer, H.~Charalambous, S.~Popescu, {\it Extremal Betti numbers and applications to monomial ideals}, J. Algebra 221, 497--512, 1999.

\bibitem[BS87]{BS}
D.~Bayer, M.~Stillman, {\it A criterion for detecting $m$-regularity}, Invent. Math. 87, 1--12, 1987.

\bibitem[BV15]{BeV}
B.~Benedetti, M.~Varbaro, {\it On the dual graph of Cohen-Macaulay algebras}, Int. Math. Res. Not., no. 17, 8085-8115, 2015.

\bibitem[Br03]{Br}
M.~Brion, \textit{Multiplicity-free subvarieties of flag varieties}, Contemp. Math. 331, 13--23, 2003.

\bibitem[BH93]{BH} W.~Bruns, J.~Herzog, \textit{Cohen-Macaulay rings}, Cambridge studies in advanced mathematics, 1993.

\bibitem[BS90]{BrSc}W.~Bruns, R.~Schw\"anzl, \textit{The number of equations defining a determinantal variety}, Bull. London Math. Soc. 22, 439--445, 1990.

\bibitem[BV88]{BV}
W.~Bruns and U.~Vetter, \emph{Determinantal rings}, Lecture Notes Mathematics 1327, Springer-Verlag, Berlin, 1988. 

\bibitem[CCV14]{CCV}
G.~Caviglia, A.~Constantinescu, M.~Varbaro, {\it On a conjecture by Kalai}, Israel J. Math. 204, 469--475, 2014.

\bibitem[CS10]{CS}
   D.~Cartwright, B.~Sturmfels,
   \emph{The Hilbert scheme of the diagonal in a product of projective spaces},
   Int. Math. Res. Not., no. 9, 1741--1771, 2010.
   
    \bibitem[Ch07]{Ch} M. Chardin, {\it Some results and questions on Castelnuovo-Mumford regularity, Syzygies and Hilbert functions}, Lect. Not. in Pure Appl. Math. 254, 1--40, 2007.
      
\bibitem[Co07]{Co} A.~Conca, 
{\em Linear spaces, transversal polymatroids and ASL domains,}
J. Alg. Combin. 25, 25--41, 2007. 

\bibitem[CDG15]{CDG1} A.~Conca, E.~De Negri, E.~Gorla, 
  \emph{Universal Gr\"obner bases for maximal minors},  
  Int. Math. Res. Not. IMRN, no. 11, 3245--3262, 2015. 
  
  \bibitem[CDG17]{CDG2} A.~Conca, E.~De Negri, E.~Gorla,	
  \emph{Multigraded generic initial ideals of determinantal ideals},
  in: Homological and Computational Methods in Commutative Algebra, 
  A. Conca, J. Gubeladze, Joseph, T. R\"omer (Eds.),
  81--96,  INdAM Series 20, Springer-Verlag, Heidelberg, 2017.

\bibitem[CDG18]{CDG3} A.~Conca, E.~De Negri, E.~Gorla,
  \emph{Universal Gr\"obner bases and Cartwright-Sturmfels ideals},
  \url{arXiv:1608.08942}, to appear in Int. Math. Res. Not. IMRN.

\bibitem[CDG18b]{CDG4} A.~Conca, E.~De Negri, E.~Gorla,
  \emph{Cartwright-Sturmfels ideals associated to graphs and linear spaces},
  \url{arXiv:1705.00575}, to appear in J. Comb. Alg. (EMS).


\bibitem[DDM18]{DDM}
H.~Dao, A.~De Stefani, L.~Ma, {\it Cohomologically full rings}, \url{arXiv:1806.00536}, 2018. 

\bibitem[DEP82]{DEP} 
C.~De Concini, D.~Eisenbud, C.~Procesi, {\it Hodge algebras}, Ast\'erisque 91, 1982.

\bibitem[DV17]{DV}
M.~Di Marca, M.~Varbaro, {\it On the diameter of an ideal}, \url{arXiv:1705.03210}, 2017.

\bibitem[Ei80]{Ei}
D.~Eisenbud, \textit{Introduction to algebras with straightening laws}, Lect. Not. in Pure and Appl. Math. 55, 243--268, 1980.

\bibitem[Ei94]{Ei1}D.~Eisenbud, \textit{Commutative Algebra with a
View Toward Algebraic Geometry}, Graduate Texts in Mathematics,
Springer, 1994.

\bibitem[EGH93]{EGH} D.~Eisenbud, M.~Green, J.~Harris, \textit{Higher Castelnuovo Theory}, Ast\'erisque 218, 187-202, 1993.

\bibitem[HS02]{HS} J.~Herzog, E.~Sbarra,  {\it Sequentially Cohen-Macaulay modules and local cohomology} 
Algebra, arithmetic and geometry, Part I, II (Mumbai, 2000) Tata Inst. Fund. Res., 327--340, 2002.
 	 


\bibitem[HR18]{HR18}
J.~Herzog, G.~Rinaldo, {\it On the extremal Betti numbers of binomial edge ideals of block graphs}, \url{arXiv:1802.06020}, 2018.

\bibitem[HW85]{HW85}
T. Hibi, K. Watanabe, {\it Study of three-dimensional algebras with straightening laws which are Gorenstein domains II}, Hiroshima Math. J. 15, 321-340, 1985.

\bibitem[Ho16]{Ho}
B.~Holmes, {\it On the diameter of dual graphs of Stanley-Reisner rings with Serre $(S_2)$ property and Hirsch type bounds on abstractions of polytopes}, \url{arXiv:1611.07354}, 2016.



\bibitem[Hi86]{Hi} T.~Hibi, {\it Every affine graded ring has a Hodge algebra structure}, Rend. Mat. Uni. Pol. Torino 44, 278--286, 1986.

\bibitem[HR76]{HR}M.~Hochster, J.~L.~Roberts \textsl{The purity of the Frobenius and Local Cohomology}, Adv. Math. 21, 117--172, 1976.

\bibitem[KS95]{KaSt} M.~Kalkbrener, B.~Sturmfels, \textsl{Initial Complex of Prime Ideals}, Adv. Math. 116, 365--376, 1995.

\bibitem[Kn09]{Kn}
A.~Knutson, {\it Frobenius splitting, point-counting, and degeneration}, \url{arXiv:0911.4941}, 2009.

\bibitem[KK18]{KK}
J.~Koll\'ar, S.~J.~Kov\'acs, {\it Deformations of log canonical singularities}, \url{arXiv:1803.03325}, 2018.

\bibitem[KK18b]{KKb}
J.~Koll\'ar, S.~J.~Kov\'acs, {\it Deformations of log canonical and F-pure singularities}, \url{arXiv:1807.07417}, 2018.


\bibitem[Ly83]{Lyu} G.~Lyubeznik, \textsl{On the Local Cohomology Modules $H_a^i(R)$ for Ideals $a$ generated by Monomials in an $R$-sequence}, Lecture Notes Mathematics 1092, 1983.

\bibitem[MSS17]{MSS} L.~Ma, K.~Schwede, K. Shimomoto, {\it Local cohomology of Du Bois singularities and applications to families}, Comp. Math. 153, 2147--2170, 2017.

\bibitem[Macaulay2]{m2}
D.~Grayson, M.~Stillman, \emph{Macaulay2, a software system for research in algebraic geometry}, available at \url{http://www.math.uiuc.edu/Macaulay2/}.

\bibitem[MS05]{MS}
E.~Miller, B.~Sturmfels, {\em Combinatorial Commutative Algebra}, Graduate Texts in Mathematics, Springer 2005.

\bibitem[Mi10]{Mi} M.~Miyazaki, \textit{On the discrete counterparts of algebras with straightening laws}, J. Comm. Alg. 2, 79--89, 2010.

\bibitem[Oh13]{Oh13}
M. Ohtani, \textit{Binomial Edge Ideals of Complete Multipartite Graphs}, Comm. Alg. 41, 3858-3867, 2013.

\bibitem[PS73]{PS}C.~Peskine, L.~Szpiro, \textit{Dimension projective finie et cohomologie locale}, IHES Publ. Math. 42, 47--119, 1973.

\bibitem[Sc79]{Sc79}
P.~Schenzel, {\em Zur lokalen Kohomologie des kanonischen Moduls}, Math. Z. 165, 223--230, 1979.

\bibitem[Sc82]{Sc} P. Schenzel, \textit{Applications of Dualiziang Complexes to Buchsbaum Rings}, Adv. Math. 44, 61--77, 1982. 

\bibitem[Sc09]{schwede} K.~Schwede, \textit{$F$-injective singularities are Du Bois}, Amer. J. Math. 131, 445--473, 2009.

\bibitem[Si99]{Si} A.K.~Singh, \textit{$F$-regularity does not deform}, Amer. J. Math. 121, 919--929, 1999.

\bibitem[St90]{St} B.~Sturmfels, \textit{Gr\"obner bases and Stanley decompositions of determinantal rings}, Math. Z. 205, 137--144, 1990.

\bibitem[St95]{sturmfels} B.~Sturmfels, \textit{Gr\"{o}bner Bases and Convex Polytopes}, University Lecture Series 8, AMS, 1995.

\bibitem[StacksProj]{SP} Stacks Project Authors, \textit{Stacks project}, available at \url{http://stacks.math.columbia.edu}.

\bibitem[Va09]{Va1}
M.~Varbaro, {\it Gr\"obner deformations, connectedness and cohomological dimension}, J. Alg. 322, 2492--2507, 2009.

\bibitem[Va18]{Va2}
M.~Varbaro, {\it Connectivity of hyperplane sections of domains}, \url{arXiv:1802.09445}, 2018.

\bibitem[Ya11]{Ya}
K.~Yanagawa, {\em Dualizing complex of the face ring of a simplicial poset}, J. Pure Appl. Alg. 215, 2231--2241, 2011.

\end{thebibliography}
\end{document}